\documentclass[final]{siamltex}

\usepackage{amsmath}
\usepackage{amsfonts}
\usepackage{booktabs}
\usepackage{leftidx}

\usepackage{tikz}
\usetikzlibrary{shapes,arrows}
\usepackage{verbatim}

\newcommand{\mA}{\mathbf{A}}
\newcommand{\mB}{\mathbf{B}}

\newcommand{\Dq}{\mathbf{D}_{\mathbf{q}_0}}

\newcommand{\mI}{\mathbf{I}}
\newcommand{\mJ}{\mathbf{J}}

\newcommand{\mL}{\mathbf{\Lambda}}
\newcommand{\mTheta}{\mathbf{\Theta}}

\newcommand{\mM}{\mathbf{M}}

\newcommand{\mQ}{\mathbf{Q}}

\newcommand{\mX}{\mathbf{X}}

\newcommand{\va}{\mathbf{a}}
\newcommand{\vb}{\mathbf{b}}

\newcommand{\ve}{\mathbf{e}}

\newcommand{\vl}{\mathbf{l}}

\newcommand{\vq}{\mathbf{q}}

\newcommand{\vx}{\mathbf{x}}

\newcommand{\ip}[1]{\left\langle #1\right\rangle}
\newcommand{\ipm}[2]{\left\langle #1\right\rangle_{#2}}
\newcommand{\norm}[2]{\left\| #1 \right\|_{#2}}
\newcommand{\submat}[2]{\left[ #1 \right]_{#2}}

\newcommand{\tvec}{\mathrm{vec}}

\newcommand{\bmat}[1]{\begin{bmatrix}#1\end{bmatrix}}

\newcommand{\mpi}{\mbox{\boldmath$\pi$}}

\title{Spectral Methods for Parameterized Matrix Equations}

\author{Paul G.~Constantine\thanks{Institute for Computational and Mathematical
Engineering,
Stanford University, Stanford, California, 94305 ({\tt paul.constantine@stanford.edu}).}
\and 
David F.~Gleich\thanks{Institute for Computational and Mathematical
Engineering,
Stanford University, Stanford, California, 94305 ({\tt dgleich@stanford.edu}).}
\and
Gianluca Iaccarino\thanks{Institute for Computational and Mathematical
Engineering, Department of Mechanical Engineering,
Stanford University, Stanford, California, 94305 ({\tt jops@stanford.edu}).}
}

\begin{document}

\maketitle

\begin{abstract}
We apply polynomial approximation methods --- known in the numerical PDEs
context as \emph{spectral methods} --- to approximate the vector-valued
function that satisfies a linear system of equations where the matrix 
and the right hand side depend on a parameter. 
We derive both an interpolatory pseudospectral method and a residual-minimizing
Galerkin method, and we show how each can be interpreted as solving a
truncated infinite system of equations; the difference between the two methods
lies in where the truncation occurs. Using classical theory, we derive 
asymptotic error estimates related to the region of analyticity of the solution,
and we present a practical residual error estimate. We verify the results with
two numerical examples.
\end{abstract}


\pagestyle{myheadings}
\thispagestyle{plain}
\markboth{P.~G. CONSTANTINE, D.~F. GLEICH, AND G.~IACCARINO}{SPECTRAL METHODS FOR MATRIX
EQUATIONS}

\section{Introduction}
We consider a system of linear equations where the elements of the matrix of
coefficients and right hand side depend analytically on a parameter.
Such systems often arise as an intermediate step within computational methods
for engineering models which depend on one or more parameters. A large class of
models employ such parameters to represent uncertainty in the input quantities;
examples include PDEs with random
inputs~\cite{Babuska04,Frauenfelder2005,Xiu02}, image deblurring
models~\cite{Chung08}, and noisy inverse problems~\cite{Chandrasekaran98}.
Other examples of parameterized linear systems occur in electronic circuit
design~\cite{Li2009}, applications of
PageRank~\cite{Brezinski06,Constantine07}, and dynamical
systems~\cite{Dieci03}. Additionally, we note a recent rational
interpolation scheme proposed by Wang et al.~\cite{Wang08} where each
evaluation of the interpolant involves a constrained least-squares problem that
depends on the point of evaluation. Parameterized linear operators have been
analyzed in their own right in the context of perturbation theory; the standard
reference for this work is Kato~\cite{Kato80}.


In our case, we are interested in approximating the vector-valued function
that satisfies the parameterized matrix equation. We will analyze the use of
polynomial approximation methods, which have evolved under the heading
``spectral methods'' in the context of numerical methods for
PDEs~\cite{Boyd01,Canuto06,Hesthaven07}. In their most basic form, these
methods are characterized by a global approximation of the function of interest by a finite
series of orthogonal (algebraic or trigonometric) polynomials. For smooth
functions, these methods converge geometrically, which is the primary reason for their popularity. The
use of spectral methods for parameterized equations is not unprecedented. In
fact, the authors were motivated primarily by the so-called polynomial chaos
methods~\cite{Ghanem91,Xiu02} and related work~\cite{Babuska04,Babuska05,Xiu05}
in the burgeoning field of uncertainty quantification. There has been some work
in the linear algebra community analyzing the fully discrete problems that
arise in this context~\cite{Ernst09,Powell08,Elman05}, but we know of no
existing work addressing the more general problem of parameterized matrix
equations. 

There is an ongoing debate in spectral methods communities surrounding the
relative advantages of Galerkin methods versus pseudospectral methods. In the
case of parameterized matrix equations, the interpolatory
pseudospectral methods only require the solution of the parameterized model
evaluated at a discrete set of points,
which makes parallel implementation straightforward. In contrast,
the Galerkin method requires the solution of a coupled linear system whose
dimension is many times larger than the original parameterized set of
equations. We offer insight into this contest by establishing a fair ground
for rigorous comparison and deriving a concrete relationship between the two
methods.

In this paper, we will first describe the parameterized matrix equation and
characterize its solution in section~\ref{sec:parmmats}. We then derive a spectral
Galerkin method and a pseudospectral method for approximating the solution to
the parameterized matrix equation in section~\ref{sec:spectral}. In
section~\ref{sec:connections}, we analyze the relationship between these methods
using the symmetric, tridiagonal Jacobi matrices -- techniques which are
reminiscent of the analysis of Gauss quadrature by Golub and
Meurant~\cite{Golub94} and Gautschi~\cite{Gautschi02}. We derive error
estimates for the methods that relate the geometric rate of convergence to the
size of the region of analyticity of the solution in section~\ref{sec:error},
and we conclude with simple numerical examples in section~\ref{sec:examples}.
See table~\ref{tab:notation} for a list of notational conventions, and note
that \emph{all index sets begin at 0 to remain consistent with the ordering of
a set of polynomials by their largest degree.}

\begin{table}
\centering
\label{tab:notation}
\caption{We attempt to use a consistent and clear notation throughout the
paper. This table details the notational conventions, which we use unless
otherwise noted. Also, all indices begin at 0. }
\begin{tabular}{clc}
\toprule
\textbf{Notation} & \textbf{Meaning}\\
\midrule
$A(s)$ & a square matrix-valued function of a parameter $s$\\
$b(s)$ & a vector-valued function of the parameter $s$\\ 
$\mA$ & a constant matrix\\ 
$\vb$ & a constant vector\\
$\ip{\cdot}$ & the integral with respect to a given weight function\\
$\ipm{\cdot}{n}$ & the integral $\ip{\cdot}$ approximated by an $n$-point Gauss
quadrature rule\\
$\submat{\mM}{r\times r}$ & the first $r\times r$ principal minor of a matrix
$\mM$\\
\bottomrule
\end{tabular}
\end{table}

\section{Parameterized Matrix Equations}
\label{sec:parmmats}
In this section, we define the specific problem we will study and characterize
its solution. We consider problems that depend on a single parameter $s$ that
takes values in the finite interval $[-1,1]$. Assume that the
interval $[-1,1]$ is equipped with a positive scalar weight function $w(s)$ such
that all moments exist, i.e.
\begin{equation}
\ip{s^k}\equiv\int_{-1}^1 s^k w(s)\,ds<\infty,\qquad k=1,2,\dots,
\end{equation}
and the integral of $w(s)$ is equal to 1. We will use the bracket notation to
denote an integral against the given weight function. In a stochastic context,
one may interpret this as an expectation operator where $w(s)$ is the density
function of the random variable $s$.

Let the $\mathbb{R}^N$-valued function $x(s)$ satisfy the linear system of
equations
\begin{equation}
\label{eq:main}
A(s)x(s)=b(s),\qquad s\in[-1,1]
\end{equation}
for a given $\mathbb{R}^{N\times N}$-valued function $A(s)$ and
$\mathbb{R}^N$-valued function $b(s)$. We assume that both $A(s)$ and $b(s)$
are analytic in a region containing $[-1,1]$, which implies that they have a
convergent power series
\begin{equation}
\label{eq:powerseries}
A(s)=\mA_0+\mA_1s+\mA_2s^2+\cdots,\qquad b(s)=\vb_0+\vb_1s+\vb_2s^2+\cdots,
\end{equation}
for some constant matrices $\mA_i$ and constant vectors $\vb_i$. Additionally,
we assume that $A(s)$ is bounded away from singularity for all $s\in[-1,1]$.
This implies that we can write $x(s)=A^{-1}(s)b(s)$.


The elements of the solution $x(s)$ can also be written using Cramer's
rule~\cite[Chapter 6]{Meyer00} as a ratio of determinants.
\begin{equation}
\label{eq:cramer}
x_i(s) = \frac{\det(A_i(s))}{\det(A(s))}, \qquad i=0,\dots,N-1,
\end{equation}
where $A_i(s)$ is the parameterized matrix formed by replacing the $i$th column
of $A(s)$ by $b(s)$. From equation \eqref{eq:cramer} and the invertibility of
$A(s)$, we can conclude that $x(s)$ is analytic in a region containing
$[-1,1]$. 
%

Equation \eqref{eq:cramer} reveals the underlying structure
of the solution as a function of $s$. If $A(s)$ and $b(s)$ depend polynomially
on $s$, then \eqref{eq:cramer} tells us that $x(s)$ is a rational function. Note
also that this structure is independent of the particular weight function $w(s)$. 

\section{Spectral Methods}
\label{sec:spectral}
In this section, we derive the spectral methods we use to approximate the
solution $x(s)$. We begin with a brief review of the relevant theory of
orthogonal polynomials, Gaussian quadrature, and Fourier series. We include this
section primarily for the sake of notation and refer the reader to a standard
text on orthogonal polynomials~\cite{Szego39} for further theoretical details
and~\cite{Gautschi04} for a modern perspective on computation.

\subsection{Orthogonal Polynomials and Gaussian Quadrature}

Let $\mathbb{P}$ be the space of real polynomials defined on $[-1,1]$,
and let $\mathbb{P}_n\subset\mathbb{P}$ be the space of polynomials of degree
at most $n$. For any $p$, $q$ in $\mathbb{P}$, we define the inner product as
\begin{equation}
\label{eq:innerproduct}
\ip{pq} \equiv \int_{-1}^1 p(s)q(s)w(s)\,ds. 
\end{equation} 
We define a norm on $\mathbb{P}$ as $\norm{p}{L^2} = \sqrt{\ip{p^2}}$, which is
the standard $L^2$ norm for the given weight $w(s)$. Let $\{\pi_k(s)\}$ be the
set of polynomials that are orthonormal with respect to 
$w(s)$, i.e.~$\ip{\pi_i\pi_j}=\delta_{ij}$. It is known that $\{\pi_k(s)\}$
satisfy the three-term recurrence relation
\begin{equation}
\label{eq:3term}
\beta_{k+1}\pi_{k+1}(s) = (s-\alpha_k)\pi_k(s) - \beta_k\pi_{k-1}(s), \qquad
k=0,1,2,\dots,
\end{equation}
with $\pi_{-1}(s)=0$ and $\pi_0(s)=1$. If we consider only the
first $n$ equations, then we can rewrite \eqref{eq:3term} as
\begin{equation}
s\pi_k(s) =
\beta_k\pi_{k-1}(s)+\alpha_k\pi_k(s)+\beta_{k+1}\pi_{k+1}(s),\qquad
k=0,1,\dots,n-1.
\end{equation}
Setting $\mpi_n(s) =
[\pi_0(s),\pi_1(s),\dots,\pi_{n-1}(s)]^T$, we can write this conveniently in
matrix form as
\begin{equation}
\label{eq:mat3term}
s\mpi_n(s) = \mJ_n\mpi_n(s) + \beta_{n}\pi_{n}(s)\ve_{n}
\end{equation}
where $\ve_n$ is a vector of zeros with a one in the last entry, and $\mJ_n$
(known as the \emph{Jacobi matrix}) is a symmetric, tridiagonal matrix defined
as
\begin{equation}
\label{eq:jacobi}
\mJ_n = 
\begin{bmatrix}
\alpha_0 & \beta_1& & &  \\
\beta_1 & \alpha_1 & \beta_2 & &  \\
 & \ddots & \ddots & \ddots &  \\
 & & \beta_{n-2} & \alpha_{n-2} & \beta_{n-1}\\
 & & & \beta_{n-1} & \alpha_{n-1}
\end{bmatrix}.
\end{equation}
The zeros $\{\lambda_i\}$ of $\pi_{n}(s)$ are the eigenvalues of $\mJ_n$ and
$\mpi_n(\lambda_i)$ are the corresponding eigenvectors; this follows directly
from \eqref{eq:mat3term}. 
Let $\mQ_n$ be the orthogonal matrix of eigenvectors of $\mJ_n$. Then we write
the eigenvalue decomposition of $\mJ_n$ as
\begin{equation}
\label{eq:eigJ}
\mJ_n = \mQ_n\mL_n\mQ_n^T.
\end{equation}
It is known (c.f.~\cite{Gautschi04}) that the eigenvalues $\{\lambda_i\}$ are the
familiar Gaussian quadrature points associated with the weight function
$w(s)$. The quadrature weight $\nu_i$ corresponding to $\lambda_i$ is equal to
the square of the first component of the eigenvector associated with
$\lambda_i$, i.e.
\begin{equation}
\mQ(0,i)^2 \;=\; \nu_i.
\end{equation}
The weights $\{\nu_i\}$ are known to be strictly positive. We will use
these facts repeatedly in the sequel. For an integrable scalar function $f(s)$,
we can approximate its integral by an $n$-point Gaussian quadrature rule, which
is a weighted sum of function evaluations,
\begin{equation}
\label{eq:gq}
\int_{-1}^1f(s)w(s)\,ds = \sum_{i=0}^{n-1} f(\lambda_i)\nu_i + R_n(f).
\end{equation}
If $f\in\mathbb{P}_{2n-1}$, then $R_n(f)=0$; that is to say the \emph{degree of
exactness} of the Gaussian quadrature rule is $2n-1$. We use the notation
\begin{equation}
\label{eq:gqnote}
\ipm{f}{n}\equiv \sum_{i=0}^{n-1} f(\lambda_i)\nu_i
\end{equation}
to denote the Gaussian quadrature rule. This is a discrete approximation to the
true integral.

\subsection{Fourier Series}

The polynomials $\{\pi_k(s)\}$ form an orthonormal basis for the Hilbert space
\begin{equation}
\label{eq:hilbert}
L^2\;\equiv\;L^2_w([-1,1])\;=\;
\left\{f:[-1,1]\rightarrow\mathbb{R}\;\left|\right.\;\norm{f}{L^2}<\infty\right\}.
\end{equation}
Therefore, any $f\in L^2$ admits a convergent \emph{Fourier series}
\begin{equation}
\label{eq:fourier}
f(s) = \sum_{k=0}^\infty \ip{f\pi_k}\pi_k(s).
\end{equation}
The coefficients $\ip{f\pi_k}$ are called the \emph{Fourier coefficients}. If we
truncate the series \eqref{eq:fourier} after $n$ terms, we are left
with a polynomial of degree $n-1$ that is the best approximation polynomial in
the $L^2$ norm. In other words, if we denote
\begin{equation}
\label{eq:truncation}
P_nf(s) = \sum_{k=0}^{n-1}\ip{f\pi_k}\pi_k(s),
\end{equation}
then
\begin{equation}
\label{eq:bestapprox}
\norm{f-P_nf}{L^2} = \inf_{p\in\mathbb{P}_{n-1}}\norm{f-p}{L^2}.
\end{equation}
In fact, the error made by truncating the series is equal to the sum of squares
of the neglected coefficients,
\begin{equation}
\label{eq:neglected}
\norm{f-P_nf}{L^2}^2 = \sum_{k=n}^\infty \ip{f\pi_k}^2.
\end{equation}
These properties of the Fourier series motivate the theory and practice of spectral
methods.  

We have shown that the each element of the solution $x(s)$ of the
parameterized matrix equation is analytic in a region containing the closed
interval $[-1,1]$. Therefore it is continuous and bounded on $[-1,1]$, which
implies that $x_i(s)\in L^2$ for $i=0,\dots,N-1$. We can thus write the
convergent Fourier expansion for each element using vector notation as
\begin{equation}
\label{eq:solfourier}
x(s) = \sum_{k=0}^\infty \ip{x\pi_k}\pi_k(s).
\end{equation}
Note that we are abusing the bracket notation here, but this will make further
manipulations very convenient. The computational strategy is to choose a
truncation level $n-1$ and estimate the coefficients of the truncated expansion. 

\subsection{Spectral Collocation}

The term \emph{spectral collocation} typically refers to the technique of
constructing a Lagrange interpolating polynomial through the exact solution
evaluated at the Gaussian quadrature points. Suppose that $\lambda_i$,
$i=0,\dots,n-1$ are the Gaussian quadrature points for the weight function $w(s)$.
We can construct an $n-1$ degree polynomial interpolant of the solution through
these points as
\begin{equation}
\label{eq:collocation}
x_{c,n}(s) \;=\; \sum_{i=0}^{n-1}x(\lambda_i)\ell_i(s) \;\equiv\;\mX_c\vl_n(s).
\end{equation}
The vector $x(\lambda_i)$ is the solution to the equation
$A(\lambda_i)x(\lambda_i)=b(\lambda_i)$. The $n-1$ degree polynomial
$\ell_i(s)$ is the standard Lagrange basis polynomial defined as
\begin{equation}
\label{eq:lagrange}
\ell_i(s) = \prod_{j=0,\;j\not=i}^{n-1} 
\frac{s-\lambda_j}{\lambda_i-\lambda_j}.
\end{equation}
The $N\times n$ constant matrix $\mX_c$ (the subscript $c$ is for
\emph{collocation}) has one column for each $x(\lambda_i)$, and $\vl_n(s)$ is a
vector of the Lagrange basis polynomials.

By construction, the collocation polynomial $x_{c,n}$ interpolates the true
solution $x(s)$ at the Gaussian quadrature points. We will use this construction
to show the connection between the pseudospectral method and the Galerkin
method.

\subsection{Pseudospectral Methods}

Notice that computing the true coefficients of the Fourier expansion of $x(s)$
requires the exact solution. The essential idea of the pseudospectral method
is to approximate the Fourier coefficients of $x(s)$ by a Gaussian quadrature
rule. In other words,
\begin{equation}
\label{eq:pseudospec}
x_{p,n}(s) \;=\; \sum_{i=0}^{n-1}\ipm{x\pi_k}{n}\pi_k(s) \;\equiv\;\mX_p\mpi_n(s),
\end{equation}
where $\mX_p$ is an $N\times n$ constant matrix of the approximated Fourier
coefficients; the subscript $p$ is for \emph{pseudospectral}. For clarity, we
recall
\begin{equation}
\label{eq:gaussfourier}
\ipm{x\pi_k}{n} = \sum_{i=0}^{n-1}x(\lambda_i)\pi_k(\lambda_i)\nu_i.
\end{equation}
where $x(\lambda_i)$ solves $A(\lambda_i)x(\lambda_i)=b(\lambda_i)$. 
In general, the number of points in the quadrature rule need not
have any relationship to the order of truncation. However, when the number of
terms in the truncated series is equal to the number of points in the
quadrature rule, the pseudospectral approximation is equivalent to the
collocation approximation. This relationship is well-known, but we include
following lemma and theorem for use in later proofs.

\begin{lemma}
\label{lem:basischange}
Let $\vq_0$ be the first row of $\mQ_n$, and define $\Dq =
\mathrm{diag}(\vq_0)$.
The matrices $\mX_p$ and $\mX_c$ are related by the equation
$\mX_p = \mX_c\Dq\mQ_n^T$.
\end{lemma}

\begin{proof}
Write
\begin{align*}
\mX_p(:,k) &= \ipm{x\pi_k}{n}\\
&= \sum_{j=0}^{n-1} x(\lambda_j)\pi_k(\lambda_j)\nu_j\\
&= \sum_{j=0}^{n-1}
\mX_c(:,j)\frac{1}{\|\mpi_n(\lambda_j)\|_2}\frac{\pi_k(\lambda_j)}{\|\mpi_n(\lambda_j)\|_2}\\
&= \mX_c\Dq\mQ_n^T(:,k)
\end{align*}
which implies $\mX_p = \mX_c\Dq\mQ_n^T$ as required.
\end{proof}

\begin{theorem}
\label{thm:pseudoequalcollocation}
The $n-1$ degree collocation approximation is equal to the $n-1$ degree
pseudospectral approximation using an $n$-point Gaussian quadrature rule, i.e.
\begin{equation}
x_{c,n}(s) = x_{p,n}(s).
\end{equation}
for all $s$.
\end{theorem}

\begin{proof}
Note that the elements of $\vq_0$ are all 
non-zero, so $\Dq^{-1}$ exists.
Then lemma \ref{lem:basischange} implies $\mX_c = \mX_p\mQ_n\Dq^{-1}$.
Using this change of variables, we can write
\begin{equation}
x_{c,n}(s) \;=\; \mX_c\vl_n(s) \;=\;
\mX_p\mQ_n\Dq^{-1}\vl_n(s).
\end{equation}
Thus it is sufficient to show that $\mpi_n(s) =
\mQ_n\Dq^{-1}\vl_n(s)$. Since this is just a vector of polynomials with
degree at most $n-1$, we can do this by multiplying each
element by each orthonormal basis polynomial up to order $n-1$ and
integrating. Towards this end we define $\mTheta \equiv \ip{\vl_n\mpi_n^T}$.

Using the polynomial exactness of the Gaussian quadrature rule, we compute
the $i,j$ element of $\mTheta$.
\begin{align*}
\mTheta(i,j) &= \ip{l_i\pi_j}\\
&= \sum_{k=0}^{n-1}
\ell_{i}(\lambda_{k})\pi_{j}(\lambda_{k})\nu_{k}\\
&= 
\frac{1}{\|\mpi_n(\lambda_{i})\|_2} 
\frac{\pi_{j}(\lambda_{i})}{\|\mpi_n(\lambda_{i})\|_2}\\
&= \mQ_n(0,i)\mQ_n(j,i),
\end{align*}
which implies that $\mTheta = \Dq\mQ_n^T$. Therefore
\begin{align*}
\ip{\mQ_n\Dq^{-1}\vl_n\mpi_n^T} &=
\mQ_n\Dq^{-1}\ip{\vl_n\mpi_n^T}\\
&= \mQ_n\Dq^{-1}\mTheta\\
&= \mQ_n\Dq^{-1}\Dq\mQ_n^T\\
&= \mI_n,
\end{align*}
which completes the proof.
\end{proof}

Some refer to the pseudospectral
method explicitly as an interpolation method~\cite{Boyd01}.
See~\cite{Hesthaven07} for an insightful interpretation in terms of a discrete
projection. Because of this property, we will freely interchange the
collocation and pseudospectral approximations when convenient in the ensuing
analysis. 

The work required to compute the pseudospectral approximation is
highly dependent on the parameterized system. In general, we assume that the
computation of $x(\lambda_i)$ dominates the work; in other words, the cost
of computing Gaussian quadrature formulas is negligible compared to computing the
solution to each linear system. Then if each $x(\lambda_i)$ costs
$\mathcal{O}(N^3)$, the pseudospectral approximation with $n$ terms costs
$\mathcal{O}(nN^3)$. 

\subsection{Spectral Galerkin}

The spectral Galerkin method computes a finite dimensional approximation to
$x(s)$ such that each element of the equation residual is orthogonal to the
approximation space. Define 
\begin{equation}
\label{eq:resid}
r(y,s) = A(s)y(s)-b(s).
\end{equation}
The finite dimensional approximation space for each component $x_i(s)$ will be
the space of polynomials of degree at most $n-1$. This space is spanned by the
first $n$ orthonormal polynomials, i.e.
$\mathrm{span}(\pi_0(s),\dots,\pi_{n-1}(s))=\mathbb{P}_{n-1}$. We seek an
$\mathbb{R}^N$-valued polynomial $x_{g,n}(s)$ of maximum degree $n-1$ such that
\begin{equation}
\label{eq:orthoresid}
\ip{r_i(x_{g,n})\pi_k}=0,\qquad i=0,\dots,N-1,\qquad k=0,\dots,n-1,
\end{equation}
where $r_i(x_{g,n})$ is the $i$th component of the residual.
We can write equations \eqref{eq:orthoresid} in matrix notation as
\begin{equation}
\label{eq:matresid}
\ip{r(x_{g,n})\mpi_n^T} = \mathbf{0}
\end{equation}
or equivalently
\begin{equation}
\label{eq:varform}
\ip{Ax_{g,n}\mpi_n^T} = \ip{b\mpi_n^T}.
\end{equation}
Since each component of $x_{g,n}(s)$ is a polynomial of degree at most $n-1$, we
can write its expansion in $\{\pi_k(s)\}$ as
\begin{equation}
\label{eq:galerkin}
x_{g,n}(s) \;=\; \sum_{k=0}^{n-1}\vx_{g,k}\pi_k(s) \;\equiv\; \mX_g\mpi_n(s),
\end{equation} 
where $\mX_g$ is a constant matrix of size $N\times n$; the subscript $g$ is
for \emph{Galerkin}. Then equation \eqref{eq:varform} becomes
\begin{equation}
\label{eq:varform2}
\ip{A\mX_g\mpi_n\mpi_n^T} = \ip{b\mpi_n^T}.
\end{equation}
Using the vec notation~\cite[Section 4.5]{GVL96}, we can rewrite
\eqref{eq:varform2} as
\begin{equation}
\label{eq:galerkinsys}
\ip{\mpi_n\mpi_n^T\otimes A}\tvec(\mX_g)=\ip{\mpi_n\otimes b}.
\end{equation} 
where $\tvec(\mX_g)$ is an $Nn\times 1$ constant vector equal to the columns of
$\mX_g$ stacked on top of each other. The constant matrix
$\ip{\mpi_n\mpi_n^T\otimes A}$ has size $Nn\times Nn$ and a distinct block
structure; the $i,j$ block of size $N\times N$ is equal to $\ip{\pi_i\pi_j A}$.
More explicitly,
\begin{equation}
\ip{\mpi_n\mpi_n^T\otimes A}
=
\bmat{
\ip{\pi_0\pi_0 A} & \cdots & \ip{\pi_{0}\pi_{n-1} A} \\
\vdots & \ddots & \vdots\\
\ip{\pi_{n-1}\pi_0 A} & \cdots & \ip{\pi_{n-1}\pi_{n-1} A} 
}.
\end{equation}
Similarly, the $i$th block of the $Nn\times 1$ vector $\ip{\mpi_n\otimes b}$
is equal to $\ip{b\pi_i}$, which is exactly the $i$th Fourier coefficient of
$b(s)$.

Since $A(s)$ is bounded and nonsingular for all $s\in[-1,1]$, it is
straightforward to show that $x_{g,n}(s)$ exists and is unique using the
classical Galerkin theorems presented and summarized in Brenner and
Scott~\cite[Chapter 2]{Brenner02}. This implies that $\mX_g$ is unique, and
since $b(s)$ is arbitrary, we conclude that the matrix
$\ip{\mpi_n\mpi_n^T\otimes A}$ is nonsingular for all finite truncations $n$. 

The work required to compute the Galerkin approximation depends on how one
computes the integrals in equation \eqref{eq:galerkinsys}. If we assume that
the cost of forming the system is negligible, then the costly part of the
computation is solving the system \eqref{eq:galerkinsys}. The size of the
matrix $\ip{\mpi_n\mpi_n^T\otimes A}$ is $Nn\times Nn$, so we expect an
operation count of $\mathcal{O}(N^3n^3)$, in general. However, many
applications beget systems with sparsity or exploitable structure that can
considerably reduce the required work.

\subsection{Summary}

We have discussed two classes of spectral methods: (i) the
interpolatory pseudospectral method which approximates the truncated Fourier
series of $x(s)$ by using a Gaussian quadrature rule to approximate each Fourier
coefficient, and (ii) the Galerkin projection method which finds an
approximation in a finite dimensional subspace such that the residual
$A(s)x_{g,n}(s)-b(s)$ is orthogonal to the approximation space. In general, the
$n$-term pseudospectral approximation requires $n$ solutions of the original
parameterized matrix equation \eqref{eq:main} evaluated at the Gaussian quadrature
points, while the Galerkin method requires the solution of the
coupled linear system of equations \eqref{eq:galerkinsys} that is $n$ times as
large as the original parameterized matrix equation. A rough operation count
for the pseudospectral and Galerkin approximations is $\mathcal{O}(nN^3)$ and
$\mathcal{O}(n^3N^3)$, respectively. 

Before discussing asymptotic error estimates, we first derive some interesting
and useful connections between these two classes of methods. In particular, we
can interpret each method as a set of functions acting on the infinite Jacobi
matrix for the weight function $w(s)$; the difference between the methods lies
in where each truncates the infinite system of equations. 

\section{Connections Between Pseudospectral and Galerkin}
\label{sec:connections}
We begin with a useful lemma for 
representing a matrix of Gauss quadrature integrals in terms of functions of
the Jacobi matrix.

\begin{lemma}
\label{lem:jacpoly}
Let $f(s)$ be a scalar function analytic in a region containing $[-1,1]$. Then
$\ipm{f\mpi_n\mpi_n^T}{n} = f(\mJ_n)$.
\end{lemma}

\begin{proof}
We examine the $i,j$ element of the $n\times n$ matrix $f(\mJ_n)$.
\begin{align*}
\ve_i^T f(\mJ_n) \ve_j &= \ve_i^T\mQ_n f(\mL_n) \mQ_n^T\ve_j\\
&= \vq_i^T f(\mL_n) \vq_j\\
&= \sum_{k=0}^{n-1} f(\lambda_k)
\frac{\pi_i(\lambda_{k})}{\|\mpi(\lambda_{k})\|_2}
\frac{\pi_j(\lambda_{k})}{\|\mpi(\lambda_{k})\|_2}\\
&= \sum_{k=0}^{n-1} f(\lambda_k) \pi_i(\lambda_k)\pi_j(\lambda_k)\nu_k^n\\
&=\ipm{f\pi_i\pi_j}{n},
\end{align*}
which completes the proof.
\end{proof}

Note that Lemma \ref{lem:jacpoly} generalizes Theorem 3.4 in~\cite{Golub94}.
With this in the arsenal, we can prove the following theorem
relating pseudospectral to Galerkin.

\begin{theorem}
\label{thm:pseudogalerkin}
The pseudospectral solution is equal to an approximation of the Galerkin
solution where each integral in equation \eqref{eq:galerkinsys} is
approximated by an $n$-point Gauss quadrature formula. In other words, $\mX_p$
solves
\begin{equation}
\label{eq:pseudospecsys}
\ipm{\mpi_n\mpi_n^T\otimes A}{n}\tvec(\mX_p) = \ipm{\mpi_n\otimes\vb}{n}.
\end{equation}
\end{theorem}

\begin{proof}
Define the $N\times n$ matrix $\mB_c=[b(\lambda_0) \cdots b(\lambda_{n-1})]$.
Using the power series expansion of $A(s)$ (equation
\eqref{eq:powerseries}), we can write the matrix of each
collocation solution as
\begin{equation}
A(\lambda_k) = \sum_{i=0}^\infty\mA_i\lambda_k^i
\end{equation}
for $k=0,\dots,n-1$. We collect these into one large block-diagonal system by
writing 
\begin{equation}
\label{eq:pf1eq1}
\left(\sum_{i=0}^\infty\mL_n^i\otimes\mA_i\right)\tvec(\mX_c)
= 
\tvec(\mB_c).
\end{equation}
Let $\mI$ be the $N\times N$ identity matrix. Premultiply \eqref{eq:pf1eq1} by
$(\Dq\otimes\mI)$, and by commutativity of diagonal matrices and the mixed
product property, it becomes
\begin{equation}
\label{eq:pf1eq2}
\left(\sum_{i=0}^\infty\mL_n^i\otimes\mA_i\right)(\Dq\otimes\mI)\tvec(\mX_c)
= 
(\Dq\otimes\mI)\tvec(\mB_c).
\end{equation}
Premultiplying \eqref{eq:pf1eq2} by $(\mQ_n\otimes\mI)$, properly inserting
$(\mQ_n^T\otimes\mI)(\mQ_n\otimes\mI)$ on the left hand side, and
using the eigenvalue decomposition \eqref{eq:eigJ}, this becomes
\begin{equation}
\label{eq:pf1eq4}
\left(\sum_{i=0}^\infty\mJ_n^i\otimes\mA_i\right)(\mQ_n\otimes\mI)(\Dq\otimes\mI)\tvec(\mX_c)
= 
(\mQ_n\otimes\mI)(\Dq\otimes\mI)\tvec(\mB_c).
\end{equation}
But note that Lemma \ref{lem:basischange} implies
\begin{equation}
(\mQ_n\otimes\mI)(\Dq\otimes\mI)\tvec(\mX_c) = \tvec(\mX_p).
\end{equation} 
Using an argument identical to the proof of Lemma \ref{lem:basischange}, we can
write
\begin{equation}
(\mQ_n\otimes\mI)(\Dq\otimes\mI)\tvec(\mB_c) = \ipm{\mpi_n\otimes b}{n}
\end{equation}
Finally, using Lemma \ref{lem:jacpoly}, equation \eqref{eq:pf1eq4} becomes
\begin{equation}
\label{eq:pf1eq5}
\ipm{\mpi_n\mpi_n^T\otimes A}{n}\tvec(\mX_p)
= 
\ipm{\mpi_n\otimes b}{n}.
\end{equation}
as required.
\end{proof}

Theorem \ref{thm:pseudogalerkin} begets a corollary giving conditions for
equivalence between Galerkin and pseudospectral approximations.

\begin{corollary}
\label{cor:equiv}
If $b(s)$ contains only polynomials of maximum degree $m_b$ and $A(s)$ contains
only polynomials of maximum degree 1 (i.e. linear functions of $s$), then
$x_{g,n}(s)=x_{p,n}(s)$ for $n\geq m_b$ for all $s\in[-1,1]$. 
\end{corollary}

\begin{proof}
The parameterized matrix $\mpi_n(s)\mpi_n(s)^T\otimes A(s)$ has polynomials of
degree at most $2n-1$. Thus, by the polynomial exactness of the Gauss quadrature
formulas,
\begin{equation}
\ipm{\mpi_n\mpi_n^T\otimes A}{n} = \ip{\mpi_n\mpi_n^T\otimes A},
\qquad
\ipm{\mpi_n\otimes b}{n} = \ip{\mpi_n\otimes b}.
\end{equation}
Therefore $\mX_g = \mX_p$, and consequently
\begin{equation}
x_{g,n}(s)\;=\;\mX_g\mpi_n(s) \;=\; \mX_p\mpi_n(s)\;=\;x_{p,n}(s).
\end{equation}
as required.
\end{proof}

By taking the transpose of equation \eqref{eq:varform2} and following the
steps of the proof of theorem \ref{thm:pseudogalerkin}, we get another
interesting corollary.

\begin{corollary}
\label{cor:aj}
First define $A(\mJ_n)$ to be the $Nn\times Nn$ constant matrix with the $i,j$
block of size $n\times n$ equal to $A(i,j)(\mJ_n)$. Next define $b(\mJ_n)$ to be
the $Nn\times n$ constant matrix with the $i$th $n\times n$ block equal to
$b_i(\mJ_n)$. Then the pseudospectral coefficients $\mX_p$ satisfy
\begin{equation}
\label{eq:pseudospectranspose}
A(\mJ_n)\tvec(\mX_p^T)=b(\mJ_n)\ve_0,
\end{equation}
where $\ve_0=[1,0,\dots,0]^T$ is an $n$-vector.
\end{corollary}


Theorem \ref{thm:pseudogalerkin} leads to a fascinating connection between the
matrix operators in the Galerkin and pseudospectral methods, namely that the
matrix in the Galerkin system is equal to a submatrix of the matrix from a
sufficiently larger pseudospectral computation. This is the key to
understanding the relationship between the Galerkin and pseudospectral
approximations. In the following
lemma, we denote the first $r\times r$ principal minor of a matrix $\mM$ by
$[\mM]_{r\times r}$.

\begin{lemma}
\label{lem:galerkinsubmat}
Let $A(s)$ contain only polynomials of degree at most $m_a$, and let $b(s)$
contain only polynomials of degree at most $m_b$. Define
\begin{equation}
\label{eq:mdef}
m\equiv m(n)\geq
\max\left(
\left\lceil\frac{m_a+2n-1}{2}\right\rceil
\;,\;
\left\lceil\frac{m_b+n}{2}\right\rceil
\right)
\end{equation}
Then
\begin{align*}
\ip{\mpi_n\mpi_n^T\otimes A} &= \submat{\ipm{\mpi_m\mpi_m^T\otimes
A}{m}}{Nn\times Nn}\\
\ip{\mpi_n\otimes b} &= \submat{\ipm{\mpi_m\otimes b}{m}}{Nn\times 1}.
\end{align*}
\end{lemma}

\begin{proof}
The integrands of the matrix $\ip{\mpi_n\mpi_n^T\otimes A}$ are polynomials of
degree at most $2n+m_a-2$. Therefore they can be integrated exactly with a Gauss
quadrature rule of order $m$. A similar argument holds for $\ip{\mpi_n\otimes
b}$.
\end{proof}

Combining Lemma \ref{lem:galerkinsubmat} with corollary \ref{cor:aj}, we get
the following proposition relating the Galerkin coefficients to the Jacobi
matrices for $A(s)$ and $b(s)$ that depend polynomially on
$s$.

\begin{proposition}
\label{prop:submatsys}
Let $m$, $m_a$, and $m_b$ be defined as in Lemma \ref{lem:galerkinsubmat}.
Define $[A]_n(\mJ_m)$ to be the $Nn\times Nn$ constant matrix with the $i,j$
block of size $n\times n$ equal to $[A(i,j)(\mJ_m)]_{n\times n}$ for
$i,j=0,\dots,N-1$. 
Define $[b]_n(\mJ_m)$ to be the $Nn\times n$ constant matrix
with the $i$th $n\times n$ block equal to $[b_i(\mJ_m)]_{n\times n}$ for
$i=1,\dots,N$. Then the Galerkin coefficients $\mX_g$ satisfy
\begin{equation}
[A]_n(\mJ_m)\tvec(\mX_g^T)=[b]_n(\mJ_m)\ve_0,
\end{equation}
where $\ve_0=[1,0,\dots,0]^T$ is an $n$-vector.
\end{proposition}

Notice that Proposition \ref{prop:submatsys} provides a way to compute the
exact matrix for the Galerkin computation without any symbolic
manipulation, but beware that $m$ depends on both $n$ and the largest degree of
polynomial in $A(s)$. Written in this form, we have no trouble taking $m$ to
infinity, and we arrive at the main theorem of this section.

\begin{theorem}
\label{thm:connections}
Using the notation of Proposition \ref{prop:submatsys} and corollary
\ref{cor:aj}, the coefficients $\mX_g$ of the $n$-term Galerkin approximation of
the solution $x(s)$ to equation \eqref{eq:main} satisfy the linear system of
equations
\begin{equation}
[A]_n(\mJ_\infty)\tvec(\mX_g^T)=[b]_n(\mJ_\infty)\ve_0,
\end{equation}
where $\ve_0=[1,0,\dots,0]^T$ is an $n$-vector.
\end{theorem}

\begin{proof}
Let $A^{(m_a)}(s)$ be the truncated power series of $A(s)$ up to order $m_a$,
and let $b^{(m_b)}(s)$ be the truncated power series of $b(s)$ up to order
$m_b$. Since $A(s)$ is analytic and bounded away from singularity for all
$s\in[-1,1]$, there exists an integer $M$ such that $A^{(m_a)}(s)$ is also
bounded away from singularity for all $s\in[-1,1]$ and all $m_a>M$ (although
the bound may be depend on $m_a$). Assume that $m_a>M$. 

Define $m$ as in equation \eqref{eq:mdef}. Then by Proposition
\ref{prop:submatsys}, the coefficients $\mX_g^{(m_a,m_b)}$ of the $n$-term
Galerkin approximation to the solution of the truncated system satisfy
\begin{equation}
\label{eq:trunceq}
[A^{(m_a)}]_n(\mJ_m)\tvec((\mX_g^{(m_a,m_b)})^T)=[b^{(m_b)}]_n(\mJ_m)\ve_0.
\end{equation} 
By the definition of $m$ (equation \eqref{eq:mdef}), equation
\eqref{eq:trunceq} holds for all integers greater than some minimum value.
Therefore, we can take $m\rightarrow\infty$ without changing the solution at
all, i.e.
\begin{equation}
[A^{(m_a)}]_n(\mJ_\infty)\tvec((\mX_g^{(m_a,m_b)})^T)=[b^{(m_b)}]_n(\mJ_\infty)\ve_0.
\end{equation}
Next we take $m_a,m_b\rightarrow\infty$ to get
\begin{align*}
[A^{(m_a)}]_n(\mJ_\infty) &\rightarrow [A]_n(\mJ_\infty)\\
[b^{(m_b)}]_n(\mJ_\infty) &\rightarrow [b]_n(\mJ_\infty)
\end{align*}
which implies
\begin{equation}
\mX_g^{(m_a,m_b)} \rightarrow \mX_g
\end{equation}
as required.
\end{proof}

Theorem \ref{thm:connections} and corollary \ref{cor:aj} reveal the fundamental
difference between the Galerkin and pseudospectral approximations. We put them
side-by-side for comparison.
\begin{equation}
\label{eq:comparison}
[A]_n(\mJ_\infty)\tvec(\mX_g^T)=b_n(\mJ_\infty)\ve_0,\qquad
A(\mJ_n)\tvec(\mX_p^T)=b(\mJ_n)\ve_0.
\end{equation}
The difference lies in where the truncation occurs. For pseudospectral, the
infinite Jacobi matrix is first truncated, and then the operator is applied.
For Galerkin, the operator is applied to the infinite Jacobi matrix, and the
resulting system is truncated. The question that remains is whether it matters.
As we will see in the error estimates in the next section, the interpolating
pseudospectral approximation converges at a rate comparable to the Galerkin
approximation, yet requires considerably less computational effort. 

\section{Error Estimates}
\label{sec:error}
Asymptotic error estimates for polynomial approximation are well-established in
many contexts, and the theory is now considered classical. Our goal is to apply
the classical theory to relate the rate of geometric convergence to some
measure of singularity for the solution. We do not seek the tightest bounds in
the most appropriate norm as in~\cite{Canuto06}, but instead we offer
intuition for understanding the asymptotic rate of convergence. 
We also present a residual error estimate that may be more useful in practice. 
We complement
the analysis with two representative numerical examples.

To discuss convergence, we need to choose a norm. In the statements and
proofs, we will use the standard $L^2$ and $L^\infty$ norms generalized to
$\mathbb{R}^N$-valued functions. 

\begin{definition}
\label{def:norm}
For a function $f:\mathbb{R}\rightarrow\mathbb{R}^N$, define the $L^2$ and
$L^\infty$ norms as
\begin{align}
\label{eq:norms}
\norm{f}{L^2} &:= \sqrt{\sum_{i=0}^{N-1} \int_{-1}^{1}f_i^2(s)w(s)\,ds}\\
\norm{f}{L^\infty} &:= \max_{0\leq i\leq N-1}\left(\sup_{-1\leq s\leq 1}
|f_i(s)|\right) 
\end{align}
\end{definition}

With these norms, we can state error estimates for both Galerkin and
pseudospectral methods.

\begin{theorem}[Galerkin Asymptotic Error Estimate]
\label{thm:galerkinerr}
Let $\rho^\ast$ be the sum of the semi-axes of the greatest ellipse with foci
at $\pm 1$ in which $x_i(s)$ is analytic for $i=0,\dots,N-1$. Then for
$1<\rho<\rho^\ast$, the asymptotic error in the Galerkin approximation is
\begin{equation}
\label{eq:galerkinerr}
\norm{x-x_{g,n}}{L^2} \leq C\rho^{-n},
\end{equation} 
where $C$ is a constant independent of $n$. 
\end{theorem}

\begin{proof}
We begin with the standard error estimate for the Galerkin
method~\cite[Section 6.4]{Canuto06} in the $L^2$ norm,
\begin{equation}
\norm{x-x_{g,n}}{L^2}\leq C \norm{x-R_nx}{L^2}.
\end{equation}
The constant $C$ is independent of $n$ but depends on the extremes of the
bounded eigenvalues of $A(s)$. Under the consistency hypothesis, the operator
$R_n$ is a projection operator such that
\begin{equation}
\norm{x_i-R_nx_i}{L^2}\rightarrow 0,\qquad n\rightarrow\infty.
\end{equation}
for $i=0,\dots,N-1$.
For our purpose, we let $R_nx$ be the expansion of $x(s)$ in terms of the
Chebyshev polynomials, 
\begin{equation}
\label{eq:chebseries}
R_nx(s) = \sum_{k=0}^{n-1}\va_kT_k(s),
\end{equation}
where $T_k(s)$ is the $k$th Chebyshev polynomial, and
\begin{equation}
\va_{k,i} = \frac{2}{\pi c_k}\int_{-1}^1 x_i(s)T_k(s)(1-s^2)^{-1/2}\,ds,\qquad 
c_k=\left\{
\begin{array}{cc}
2 & \mbox{ if $k=0$}\\
1 & \mbox{ otherwise}
\end{array}
\right.
\end{equation}
for $i=0,\dots,N-1$.
Since $x(s)$ is continuous for all $s\in[-1,1]$ and $w(s)$ is normalized, we
can bound
\begin{equation}
\norm{x-R_nx}{L^2} \leq \sqrt{N}\norm{x-R_nx}{L^\infty}
\end{equation}
The Chebyshev series
converges uniformly for functions that are continuous on
$[-1,1]$, so we can bound
\begin{align}
\norm{x-R_nx}{L^\infty} &= \norm{\sum_{k=n}^\infty\va_kT_k(s)}{L^\infty}\\
 &\leq \norm{\sum_{k=n}^\infty|\va_k|}{\infty}
\end{align}
since $-1\leq T_k(s)\leq 1$ for all $k$. To be sure, the quantity $|\va_k|$ is
the component-wise absolute value of the constant vector $\va_k$, and the norm
$\|\cdot\|_\infty$ is the standard infinity norm on $\mathbb{R}^N$.

 Using the classical
result stated in~\cite[Section 3]{Gottlieb77}, we have
\begin{equation}
\limsup_{k\rightarrow\infty}|\va_{k,i}|^{1/k}=\frac{1}{\rho^\ast_i},\qquad
i=0,\dots,N-1
\end{equation}
where $\rho^\ast_i$ is the sum of the semi-axes of the greatest ellipse with
foci at $\pm 1$ in which $x_i(s)$ is analytic. This implies that asymptotically
\begin{equation}
|\va_{k,i}|=\mathcal{O}\left(\frac{\rho_i}{k}\right),\qquad
i=0,\dots,N-1.
\end{equation}
for $\rho_i<\rho^\ast_i$. We take $\rho=\min_i \rho_i$, which suffices to prove
the estimate \eqref{eq:galerkinerr}.  
\end{proof}

Theorem \ref{thm:galerkinerr} recalls the well-known fact that the convergence
of many polynomial approximations (e.g.~power series, Fourier series) depend
on the size of the region in the complex plane in which the function is
analytic. Thus, the location of the singularity nearest the interval $[-1,1]$
determines the rate at which the approximation converges as one includes
higher powers in the polynomial approximation. Next we derive a similar result
for the pseudospectral approximation using the fact that it interpolates
$x(s)$ at the Gauss points of the weight function $w(s)$. 


\begin{theorem}[Pseudospectral Asymptotic Error Estimate]
\label{thm:pseudospecerr}
Let $\rho^\ast$ be the sum of the semi-axes of the greatest ellipse with foci
at $\pm 1$ in which $x_i(s)$ is analytic for $i=0,\dots,N-1$. Then for
$1<\rho<\rho^\ast$, the asymptotic error in the pseudospectral approximation is 
\begin{equation}
\label{eq:pseudospecerr}
\norm{x-x_{p,n}}{L^2} \leq C\rho^{-n},
\end{equation} 
where $C$ is a constant independent of $n$. 
\end{theorem}

\begin{proof}
Recall that $x_{c,n}(s)$ is the Lagrange interpolant of $x(s)$ at the Gauss
points of $w(s)$, and let $x_{c,n,i}(s)$ be the $i$th component of
$x_{c,n}(s)$. We will use the result from~\cite[Theorem 4.8]{Rivlin69} that
\begin{equation}
\label{eq:lagrangeerr}
\int_{-1}^1 (x_i(s) - x_{c,n,i}(s))^2w(s)\,ds
\leq
4E^2_n(x_i),
\end{equation}
where $E_n(x_i)$ is the error of the best approximation polynomial in the
uniform norm. We can, again, bound $E_n(x_i)$ by the error of the Chebyshev
expansion \eqref{eq:chebseries}. Using Theorem \ref{thm:pseudoequalcollocation}
with equation \eqref{eq:lagrangeerr},
\begin{align*}
\norm{x-x_{p,n}}{L^2} &= \norm{x-x_{c,n}}{L^2}  \\
&\leq 2\sqrt{N}\norm{x-R_nx}{L^\infty}.
\end{align*}
The remainder of the proof proceeds exactly as the proof of theorem
\ref{thm:galerkinerr}. 
\end{proof}

We have shown, using classical approximation theory, that the interpolating
pseudospectral method and the Galerkin method have the same asymptotic rate of
geometric convergence. This rate of convergence depends on the size of the
region in the complex plane where the functions $x(s)$ are analytic. The
structure of the matrix equation reveals at least one singularity that occurs
when $A(s^\ast)$ is rank-deficient for some $s^\ast\in\mathbb{R}$, assuming the
right hand side $b(s^\ast)$ does not fortuitously remove it. For a general
parameterized matrix, this fact may not be useful. However, for many
parameterized systems in practice, the range of the parameter is dictated by
existence and/or stability criteria. The value that makes the system singular
is often known and has some interpretation in terms of the model. In these
cases, one may have an upper bound on $\rho$, which is the sum of the
semi-axes of the ellipse of analyticity, and this can be used to estimate the
geometric rate of convergence \emph{a priori}. 

We end this section with a residual error estimate -- similar to residual error
estimates for constant matrix equations -- that may be more useful in practice
than the asymptotic results.

\begin{theorem}
\label{thm:residerr}
Define the residual $r(y,s)$ as in equation \eqref{eq:resid}, and let
$e(y,s)=x(s)-y(s)$ be the $\mathbb{R}^N$-valued function representing the error
in the approximation $y(s)$. Then
\begin{equation}
\label{eq:residerr}
C_1\norm{r(y)}{L^2}\leq\norm{e(y)}{L^2}\leq C_2\norm{r(y)}{L^2}
\end{equation}
for some constants $C_1$ and $C_2$, which are independent of $y(s)$. 
\end{theorem}

\begin{proof}
Since $A(s)$ is non-singular for all $s\in[-1,1]$, we can write
\begin{equation}
A^{-1}(s)r(y,s)\;=\;y(s)-A^{-1}(s)b(s)\;=\;e(y,s)
\end{equation}
so that
\begin{align*}
\norm{e(y)}{L^2}^2 &= \ip{e(y)^Te(y)}\\
&= \ip{r^T(y)A^{-T}A^{-1}r(y)}\\
\end{align*}
Since $A(s)$ is bounded, so is $A^{-1}(s)$. Therefore, there exist constants
$C_1^\ast$ and $C_2^\ast$ that depend only on $A(s)$ such that
\begin{equation}
C_1^\ast\ip{r^T(y)r(y)}\leq\ip{e^T(y)e(y)}\leq C_2^\ast\ip{r^T(y)r(y)}.
\end{equation}
Taking the square root yields the desired result.
\end{proof}

Theorem \ref{thm:residerr} states that the $L^2$ norm of the residual behaves
like the $L^2$ norm of the error. In many cases, this residual error may be
much easier to compute than the true $L^2$ error. However, as in residual error
estimates for constant matrix problems, the constants in Theorem
\ref{thm:residerr} will be large if the bounds on the eigenvalues of $A(s)$
are large.  We apply these results in the next section with two numerical
examples.

\section{Numerical Examples}
\label{sec:examples}
We examine two simple examples of spectral methods applied to parameterized
matrix equations. The first is a $2\times 2$ symmetric parameterized matrix,
and the second comes from a discretized second order ODE. In both cases, we
relate the convergence of the spectral methods to the size of the region of analyticity
and verify this relationship numerically. We also compare the behavior of the
true error to the behavior of the residual error estimate from theorem
\ref{thm:residerr}. 

To keep the computations simple, we use a constant weight function $w(s)$.
The corresponding orthonormal polynomials are the normalized Legendre
polynomials, and the Gauss points are the Gauss-Legendre points.

\subsection{A $2\times 2$ Parameterized Matrix Equation}

Let $\epsilon>0$, and consider the following parameterized matrix equation
\begin{equation}
\label{eq:2x2}
\bmat{1+\varepsilon & s\\ s & 1}\bmat{x_0(s) \\ x_1(s)} = \bmat{2\\ 1}.
\end{equation}
For this case, we can easily compute the exact solution,
\begin{equation}
x_0(s)=\frac{2-s}{1+\varepsilon-s^2},\qquad
x_1(s)=\frac{1+\varepsilon-2s}{1+\varepsilon-s^2}.
\end{equation}
Both of these functions have a poles at $s=\pm\sqrt{1+\varepsilon}$, so the
sum of the semi-axes of the ellipse of analyticity is bounded, i.e. 
$\rho<\sqrt{1+\varepsilon}$. Notice that the matrix is linear in $s$, and the
right hand side has no dependence on $s$. Thus, corollary \ref{cor:equiv}
implies that the Galerkin approximation is equal to the pseudospectral
approximation for all $n$; there is no need to solve the system
\eqref{eq:galerkinsys} to compute the Galerkin approximation. 
In figure \ref{fig:ex1} we plot both the true $L^2$ error and the residual error
estimate for four values of $\varepsilon$. The results confirm the analysis.

\begin{figure}
\label{fig:ex1}
\begin{center}
\includegraphics[scale=0.38]{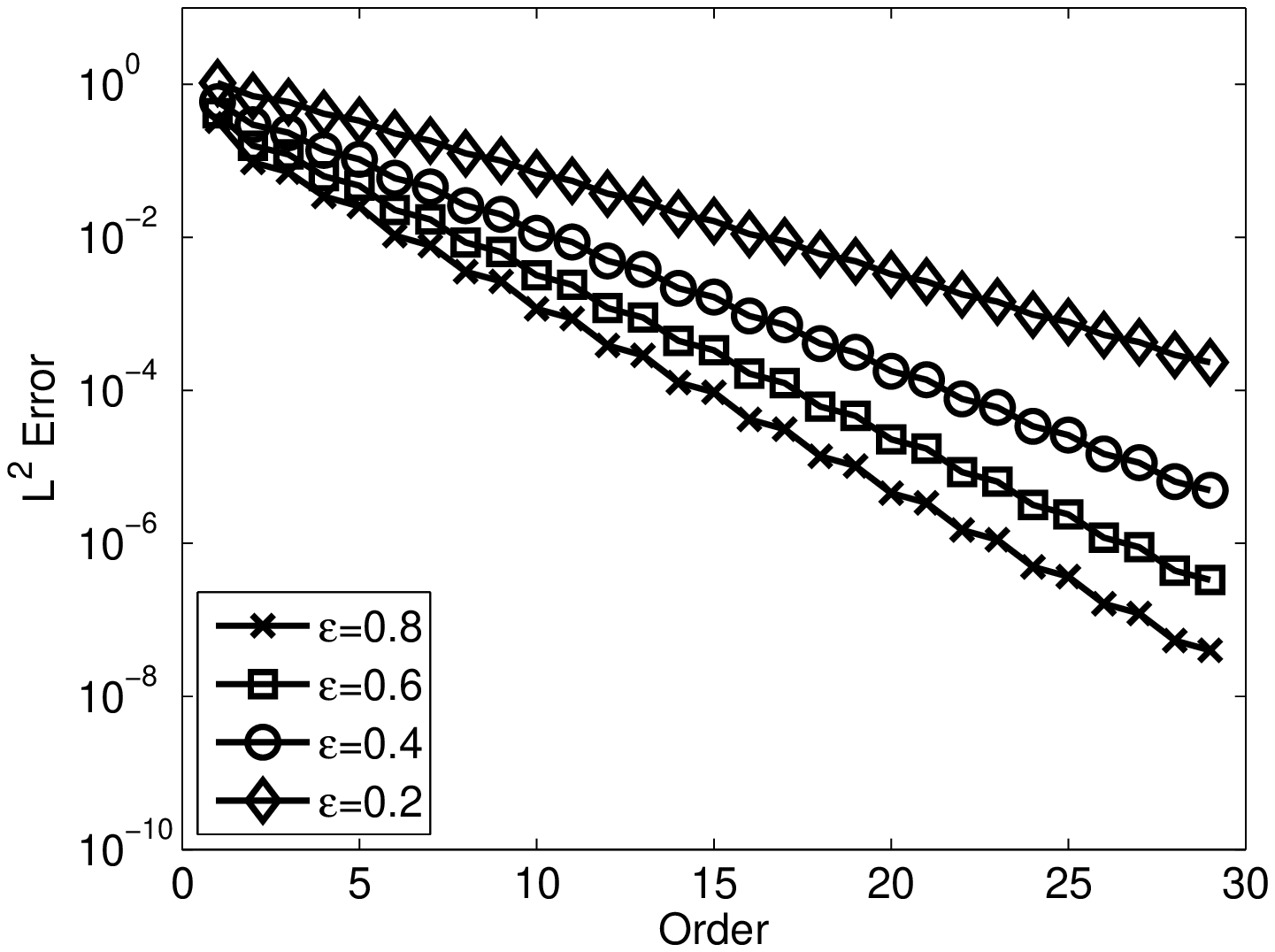}
\quad
\includegraphics[scale=0.38]{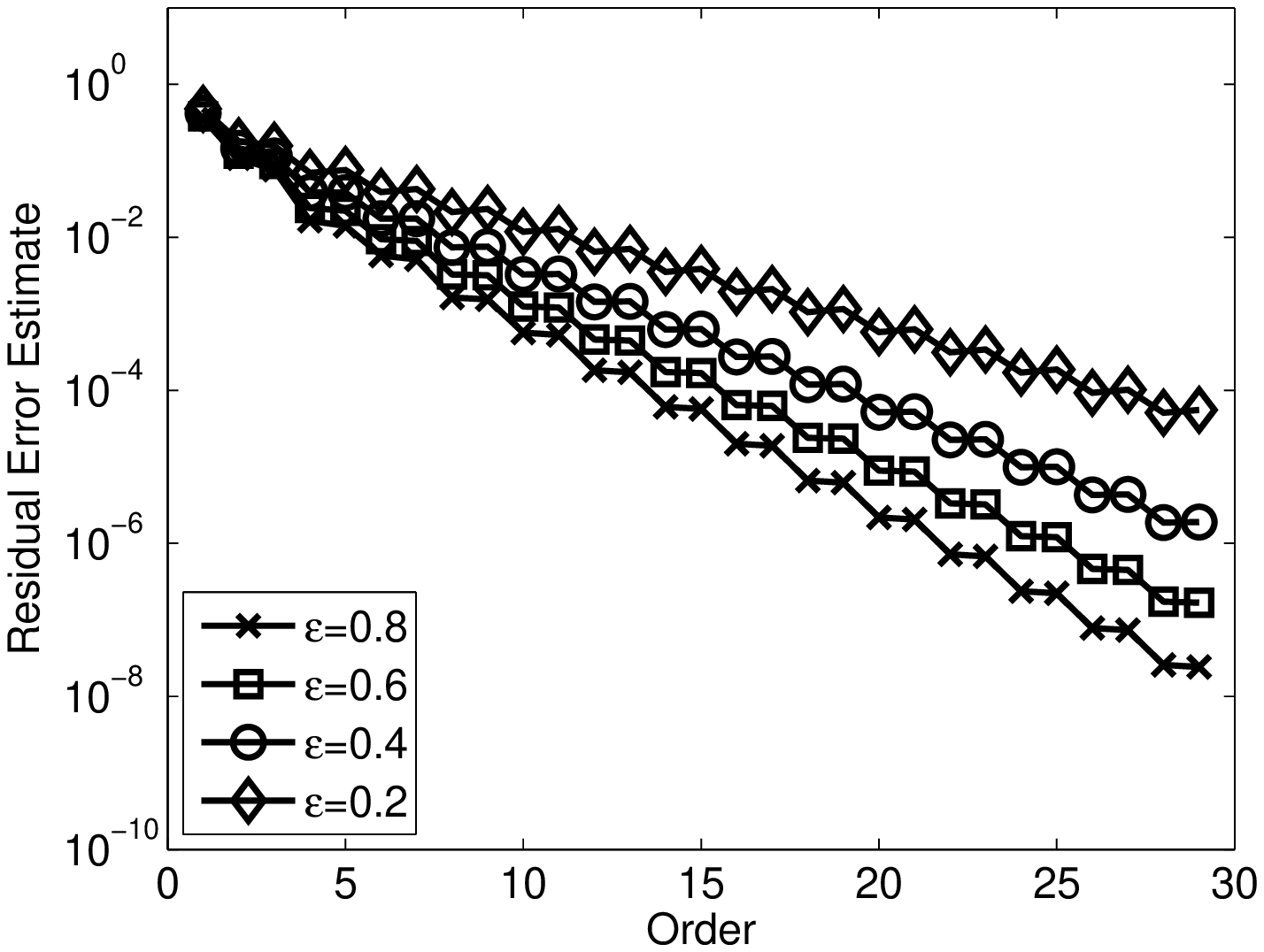}
\end{center}
\caption{The convergence of the spectral methods applied to equation
\eqref{eq:2x2}. The figure on the left shows plots the $L^2$ error as the order
of approximation increases, and the figure on the right plots the residual
error estimate. The stairstep behavior relates to the fact that
$x_0(s)$ and $x_1(s)$ are odd functions over $[-1,1]$.}
\end{figure}

\subsection{A Parameterized Second Order ODE}

Consider the second order boundary value problem
\begin{align}
\label{eq:ode}
\frac{d}{dt}\left(
\alpha(s,t)\frac{du}{dt}
\right)
&=
1\qquad t\in[0,1]\\
u(0)&=0\\
u(1)&=0
\end{align}
where, for $\varepsilon>0$,
\begin{equation}
\label{eq:coefficients}
\alpha(s,t) = 1+4\cos(\pi s)(t^2-t),\qquad s\in[\varepsilon,1].
\end{equation}
The exact solution is 
\begin{equation}
\label{eq:exactsol}
u(s,t) = \frac{1}{8\cos(\pi s)}\ln\left(1+4\cos(\pi s)(t^2-t)\right).
\end{equation}
The solution $u(s,t)$ has a singularity at $s=0$ and $t=1/2$. Notice that we
have adjusted the range of $s$ to be bounded away from 0 by $\varepsilon$.
We use a standard piecewise linear Galerkin finite element method with $512$
elements in the $t$ domain to construct a stiffness matrix parameterized by
$s$, i.e.
\begin{equation}
\label{eq:femat}
(K_0+\cos(\pi s)K_1)x(s) = b.
\end{equation}
Figure \ref{fig:ex2} shows the convergence of the residual error estimate for both
Galerkin and pseudospectral approximations as $n$ increases. (Despite having the
exact solution \eqref{eq:exactsol} available, we do not present the decay of the
$L^2$ error; it is dominated entirely by the discretization error in the $t$
domain.) As
$\varepsilon$ gets closer to zero, the geometric convergence rate of the 
spectral methods degrades considerably. Also, note that each element of the
parameterized stiffness matrix is an analytic function of $s$, but figure
\ref{fig:ex2} verifies that the less expensive pseudospectral approximation
converges at the same rate as the Galerkin approximation.

\begin{figure}
\label{fig:ex2}
\begin{center}
\includegraphics[scale=0.38]{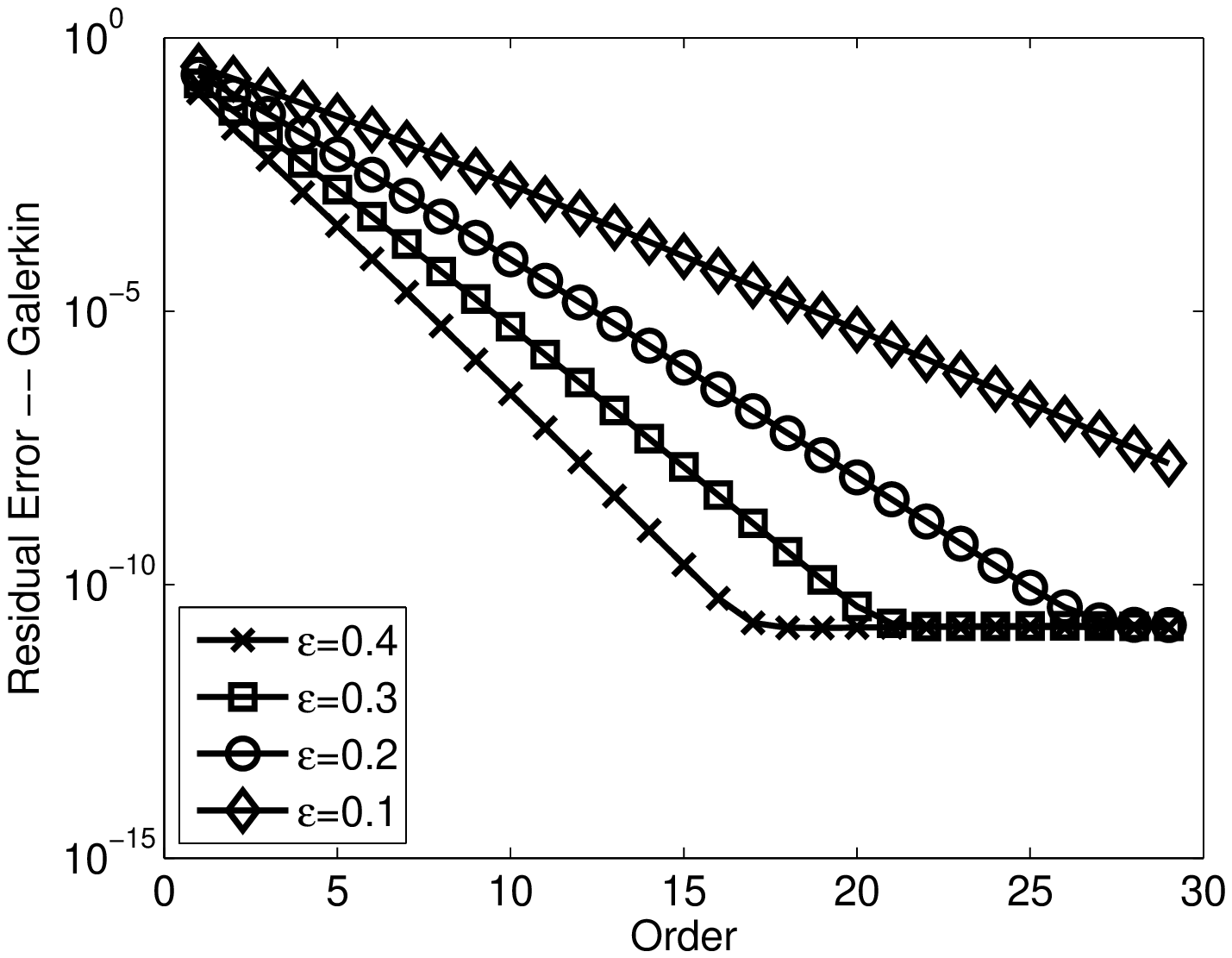}
\quad
\includegraphics[scale=0.38]{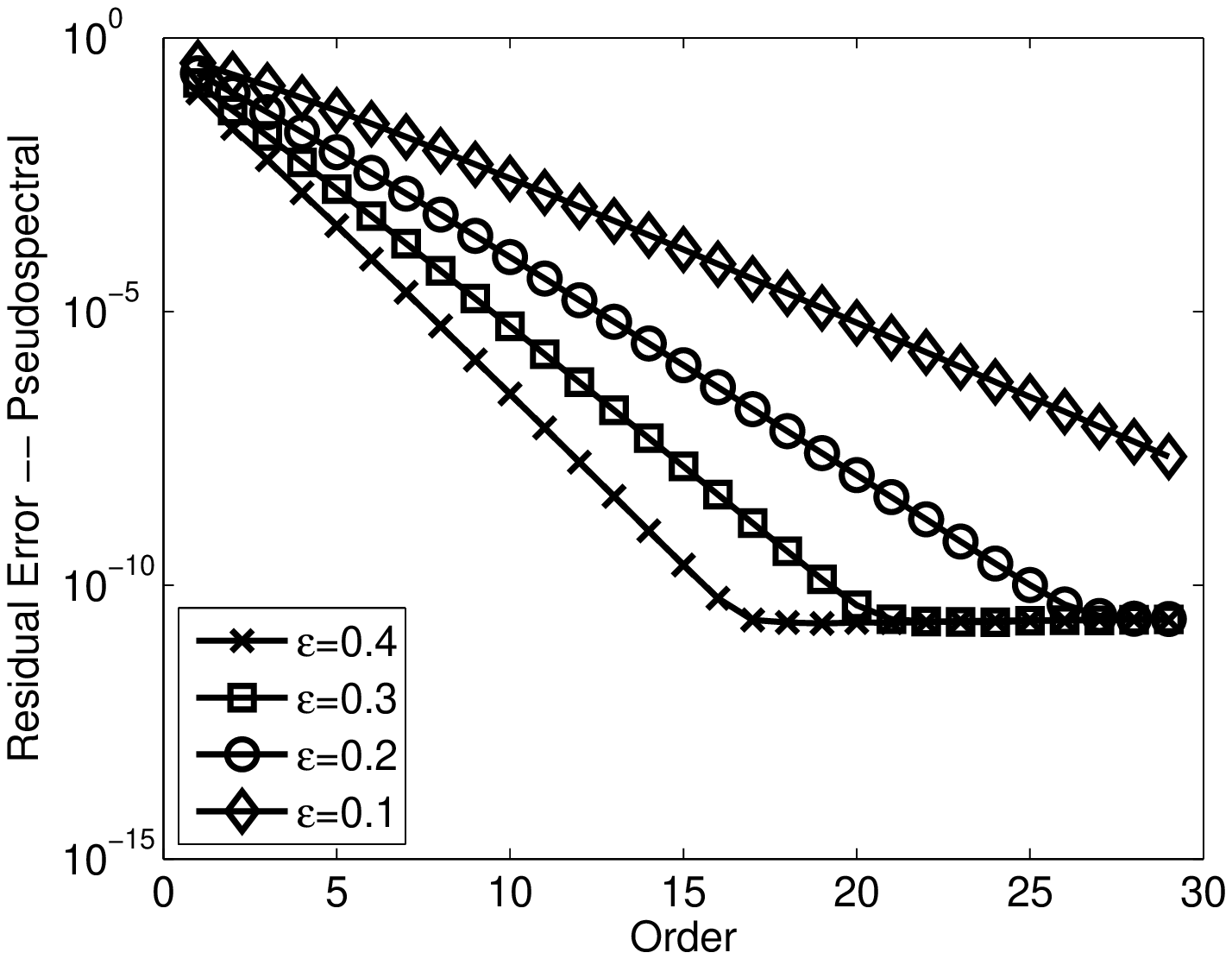}
\end{center}
\caption{The convergence of the residual error estimate for the Galerkin and
pseudospectral approximations applied to the parameterized matrix equation
\eqref{eq:femat}.}
\end{figure}

\section{Summary and Conclusions}
\label{sec:summary}
We have presented an application of spectral methods to parameterized matrix
equations. Such parameterized systems arise in many applications. The goal of
a spectral method is to construct a global polynomial approximation of the
$\mathbb{R}^N$-valued function that satisfies the parameterized system.  

We derived two basic spectral methods: (i) the interpolatory pseudospectral
method, which approximates the coefficients of the truncated Fourier series
with Gauss quadrature formulas, and (ii) the Galerkin method, which finds an
approximation in a finite dimensional subspace by requiring that the equation
residual be orthogonal to the approximation space. The primary work involved in
the pseudospectral method is solving the parameterized system at a finite set
of parameter values, whereas the Galerkin method requires the solution of a
coupled system of equations many times larger than the original parameterized
system. 

We showed that one can interpret the differences between these
two methods as a choice of when to truncate an infinite linear system of
equations. Employing this relationship we derived conditions under which these
two approximations are equivalent. In this case, there is no reason to solve the
large coupled system of equations for the Galerkin approximation.

Using classical techniques, we presented asymptotic error estimates
relating the decay of the error to the size of the region of analyticity of the
solution; we also derived a residual error estimate that may be more useful in
practice. We verified the theoretical developments with two numerical examples:
a $2\times 2$ matrix equation and a finite element discretization of a
parameterized second order ODE. 

The popularity of spectral methods for PDEs stems from their \emph{infinite}
(i.e. geometric) order of convergence for smooth functions compared to finite
difference schemes. We have the same advantage in the case of parameterized
matrix equations, plus the added bonus that there are no boundary conditions
to consider. The primary concern for these methods is determining the value of
the parameter closest to the domain that renders the system singular. 

\section{Acknowledgements}

We would like to thank James Lambers for his helpful and insightful feedback.
The first and third authors were funded by the Department of Energy Predictive
Science Academic Alliance Program and the second author was supported
by a Microsoft Live Labs Fellowship.

\bibliographystyle{siam}
\bibliography{paulconstantine}

\end{document}